\documentclass[12pt]{amsart}
\usepackage{amsmath,amsfonts,amssymb,amsthm,amstext,pgf,graphicx,hyperref,verbatim,lmodern,textcomp,color,young,tikz}
\usetikzlibrary{decorations}
\usepackage[mathscr]{euscript}
\usetikzlibrary{decorations.markings}
\usetikzlibrary{arrows}
\theoremstyle{plain}
\newtheorem{theorem}{Theorem}[section]
\newtheorem{lemma}[theorem]{Lemma}

\newtheorem{corollary}{Corollary}
\newtheorem{rem}{Remark}

\theoremstyle{definition}
\newtheorem{defn}{Definition}[section]

\title{}
\begin{document}
\title [generalized power sum and newton-girard identities] {generalized power sum and newton-girard identities}
\author[Sudip Bera]{Sudip Bera}
\author[Sajal Kumar Mukherjee]{Sajal Kumar Mukherjee}
\address[Sudip Bera]{Department of Mathematics, Indian Institute of Science, Bangalore 560 012, India.}
\email{sudipbera@iisc.ac.in}
\address[Sajal Kumar Mukherjee]{Department of Mathematics,  Indian Institute of Science, Bangalore 560 012, India.}
\email{sajalm@iisc.ac.in}

\maketitle
\begin{abstract}
In this article we prove an algebraic identity which significantly generalizes the formula for sum of powers of consecutive integers involving Stirling numbers of the second kind. Also we have obtained a generalization of Newton-Girard power sum identity.
\end{abstract}
\section{Introduction}
The sum of powers of consecutive integers has a long and fascinating history. Historically the first ever formula for the sum was obtained by Swiss mathematician 
Jacob Bernoulli (1654-1705), who proved the following:
\begin{equation}\label{power sum bernoli}
1^m+2^m+\cdots+ (n-1)^m=\frac{1}{m+1}\sum\limits_{k=0}^m\binom{m+1}{k}B_kn^{m+1-k}, m\geq 0, n\geq 1,
\end{equation}   
where $B_k^{'s}$ are the famous Bernoulli numbers. There is also a surprising relationship between the sum of powers and the Stirling numbers of second kind \cite{FQ}. In fact, 
\begin{equation}\label{power sum  stirling}
1^m+2^m+\cdots+ n^m=\frac{1}{m+1}\sum\limits_{k=0}^n\binom{n+1}{k+1}S(m,k)k!, \text{ where }
\end{equation}
\begin{equation}\label{stirling no}
 S(m, k)=\frac{1}{k!}\sum\limits_{j=1}^k (-1)^{k-j}\binom{k}{j}j^m.
\end{equation}
In \cite{FQ}, the author proved \eqref{power sum  stirling} along with its many generalizations using the so called binomial transform. In fact, the author proved the following general statement, and obtained various power sum identities as a corollary of the following:
\begin{lemma}[Lemma 2.1, \cite{FQ}]\label{fq lemma}
Let $c_1, c_2, \cdots, $ be a sequence of complex numbers. Then for every positive integer $m$, we have 
\begin{equation}
\sum\limits_{k=1}^{m}k^{\alpha}c_k=\sum\limits_{j=1}^{m}j!S(\alpha, j)\sum\limits_{k=j}^m \binom{k}{j}c_k.
\end{equation} 	
\end{lemma}
In this article, we prove a general identity, which proves Lemma \ref{fq lemma}, for positive integer $\alpha$ as a corollary and consequently many other well known power sum identities. Before stating our result, let us fix some notations. 

Let \[\{x_i^{(j)}: 1\leq j\leq r, 1\leq i\leq m\} \text{ and } \{y_{\ell}, 2\leq \ell \leq m+1\}\] be two sets of variables, and $P\subset [n] (\text{where } [n]=\{1, 2, \cdots, n\}).$ Define \[\Pi_r P=\prod_{j=1}^r\left(\sum\limits_{i\in P}x_i^{(j)}\right),\] and for any finite set $Q$ of positive integers, the maximum element of $Q$ is denoted by Max($Q$). Then we have the following:
\begin{theorem}\label{main thm}
\[\sum\limits_{k=1}^m \Pi_r[k]y_{k+1}=\sum\limits_{U\subset [m+1], |U|\geq 2}\left(\sum\limits_{\varnothing \neq V\subset U\setminus \{\text{Max}(U)\}}(-1)^{|U|-|V|-1}\Pi_rV\right)y_{\text{Max}(U)}.\]	
\end{theorem} 
 We call this theorem ``the generalized power sum theorem''. Note that, if we put $x_i^{(j)}=1$ for all $1\leq j\leq r, 1\leq i\leq m,$ and $y_{\ell}=c_{\ell-1}$ and $r=\alpha$ in our Theorem \ref{main thm}, we obtain Lemma \ref{fq lemma} for positive integer $\alpha.$ In particular, if we put $x^{(j)}_i=1=y_{\ell}$ for all $1\leq j\leq r, 1\leq i\leq m$ and $2\leq \ell\leq m+1,$ we obtain the classical formula for the sum of powers \eqref{power sum  stirling}.  
 
Newton-Girard identity is a very important result occurring many places in algebra and combinatorics. A combinatorial proof of Newton-Girard identity was first given by Doron Zeilberger in \cite{11}. In \cite {NEWGIRARD}, the present authors gave a graph theoretic formulation of the Newton-Girard identity exhibiting a relation between weighted sum of closed walks and weighted sum of linear subdigraphs (defined later) of a weighted digraph $\Gamma.$ In this paper, based on \cite {NEWGIRARD}, we give a ``colored'' version of the graph theoretic formulation mentioned above and as a corollary we obtain a significant generalization of the classical Newton-Girard identity. Before proceeding to the statement of our theorem, let us define some graph theoretic notions. 
 \begin{defn}
A weighted \emph{$k$-colored digraph}, denoted by $\Gamma_{k, C}$ is a digraph equipped with a finite set $C=\{c_1, c_2, \cdots, c_k\}$ called the set of \emph{colors} such that for any ordered pair of vertices $(i, j)$ in $\Gamma_{k, C},$ either there are no directed edges from $i$ to $j$ or there are precisely $k$ directed edges from $i$ to $j$ each receiving distinct colors  from  $C$ (i.e. no two of the directed edges from $i$ to $j$ receive the same color) and the edge from $i$ to $j$ with color $c_r$ is assigned a nonzero weight $a_{ij}^{(r)}.$ 
\end{defn}
\begin{defn}
A \emph{colored linear subdigraph} $\gamma$ of a weighted $k$-colored digraph $\Gamma_{k, C}$ is a collection of pairwise vertex-disjoint cycles such that any two different edges of $\gamma$ have different colors. Define $\text{Col}(\gamma)$ to be the set of colors of the edges in $\gamma.$ $L(\gamma)$ is defined to be the length of $\gamma$ i.e. the number of edges in $\gamma.$ 
The weight of a colored linear subdigraph $\gamma,$ written as $W(\gamma)$, is the product of the weights of all its edges. The number of cycles contained in $\gamma$ is denoted by $c(\gamma).$ The set of all colored linear subdigraphs of $\Gamma_{k, C}$ is denoted by $CLSD(\Gamma_{k, C}).$ 		
\end{defn}
\begin{rem}
Colored digraphs have appeared in the literature before. See, for example \cite{MIXEDDIS}. Here we have defined this notion in a way, suitable to our purpose.  	
\end{rem}	
\begin{defn}
A \emph{colored closed walk} $w$ of length $L(w)=m$ in a weighted $k$-colored digraph $\Gamma_{k, C}$ is a sequences of vertices $x_0, x_1, \cdots, x_{m-1}, x_m$ such that $x_0=x_m$ and for each $0\leq i\leq m-1,$ there is a directed edge from $x_i$ to $x_{i+1},$ and for $i\neq j$ the color of the directed edge from $x_i$ to $x_{i+1}$ is distinct from that of the directed edge from $x_j$ to $x_{j+1}.$ Define $\text{Col}(w)$ to be the set of colors of the edges in $w.$ The weight $W(w)$ of a colored closed walk $w$ is the product of all weights of the edges present in that walk. The set of all colored closed walks of $\Gamma_{k, C}$ is denoted by $CCW(\Gamma_{k, C}).$
\end{defn}
Let $\Gamma_{k, C}$ be a weighted $k$-colored digraph and $S$ and $T$ be subsets of $C$ such that $|S|=p$ and $|T|=q.$  We define the following:
\begin{equation*}\label{weighted color lsd}
\ell_{p, S}\triangleq
\begin{cases}
\sum\limits_{\substack{\gamma\in CLSD(\Gamma_{k, C})\\ \text{ such that }\\ L(\gamma)=p,\\ \text{ and Col}(\gamma)=S}} (-)^{c(\gamma)}W(\gamma), & \text{ if } S\neq \varnothing  \\
1, & \text{ if } S =\varnothing \text{ or equivalently } p=0\\
0,&\text{ if there  } \nexists \text{ any } \gamma \in CLSD\text{ such that } \text{ Col}(\gamma)=S.
\end{cases}
\end{equation*}
\begin{equation*}\label{weighted color closed walk}
c_{q, T}\triangleq
\begin{cases}
\sum\limits_{\substack{w\in CCW(\Gamma_{k, C})\\ \text{ such that }\\ L(w)=q,\\ \text{ and Col}(w)=T}} W(w), & \text{ if } T\neq \varnothing  \\
1, & \text{ if } T =\varnothing \text{ or equivalently } q=0\\
0,&\text{ if there  } \nexists \text{ any } w \in CCW\text{ such that } \text{ Col}(w)=T.
\end{cases}
\end{equation*}
Now we have the following:
\begin{theorem}\label{colored nwtn thm}
Let $\Gamma_{k, C}$ be a weighted $k$-colored digraph. Then 
\begin{enumerate}
 \item 
 $\sum\limits_{p+q=r, S\cap T=\varnothing}c_{q, T}\ell_{p, S}=0$, if $r>n $
 \item
 $\sum\limits_{p+q=r, q>0, S\cap T=\varnothing}c_{q, T}\ell_{p, S}+r\ell_{r, C}=0$, if $r\leq n.$
 \end{enumerate}  	
 \end{theorem}
 As a corollary of this theorem we obtain the following very generalized form of classical Newton-Girard identity.	
 \begin{theorem}\label{col newtn identity}
 Let $r$ and $n$ be two positive integers and $\{\alpha^{(i)}_j:1\leq j\leq n, 1\leq i\leq r\}$ be a set of variables. Then our theorem states the following: 	
 \begin{enumerate}
 	\item
If $r>n, \sum\limits_{k=0}^{r}(-1)^k\sum \limits_{1\leq i_1<i_2\cdots<i_{r-k}\leq r}(r-k)!\left(\sum\limits_{j=1}^n\alpha_j^{(i_1)} \alpha_j^{(i_2)}\cdots \alpha_j^{(i_{r-k})}\right)X=0, $
 \item
If $r\leq n, \sum\limits_{k=0}^{r-1}(-1)^k\sum \limits_{1\leq i_1<i_2\cdots<i_{r-k}\leq r}(r-k)!\left(\sum\limits_{j=1}^n\alpha_j^{(i_1)} \alpha_j^{(i_2)}\cdots \alpha_j^{(i_{r-k})}\right)  X+rY=0,$
 \end{enumerate}
 where \begin{align*}
 &X=\sum\limits_{\substack{(i_1^{'}, \cdots, i_k^{'}) \\ \text{ such that each }\\ i_{\ell}^{'}\in [r]\setminus \{i_1, \cdots, i_{r-k}\}\\ \text{ and } i_p^{'}\neq i_q^{'} \text{ for }p\neq q}}\sum\limits_{1\leq j_1<j_2<\cdots<j_k\leq n}\alpha_{j_1}^{(i_1^{'})}\alpha_{j_2}^{(i_2^{'})}\cdots \alpha_{j_k}^{(i_k^{'})}, \text{ and }\\
 &Y= \sum\limits _{\substack{(i_1^{'}, \cdots, i_r^{'})\\ \text{ such that each } \\i^{'}_{p} \in [r] \text{ and }\\ i^{'}_{p}\neq i^{'}_{q}  \text{ for } p\neq q}} \sum\limits_{1\leq j_1<j_2<\cdots< j_r\leq n} \alpha^{(i_1^{'})}_{j_1}\alpha^{(i_2^{'})}_{j_2}\cdots\alpha^{(i_r^{'})}_{j_r}.
 \end{align*}	
 \end{theorem}
 Note that if we put $\alpha^{(i)}_j=\alpha_j,$ for all $1\leq j\leq n$ and $1\leq i\leq r$ in Theorem \ref{col newtn identity}, we immediately obtain the following: 
 \begin{corollary}[Newton-Girard identity]
Let $\alpha_1, \alpha_2, \ldots, \alpha_n$ be roots of the polynomial $f(x)=x^n+e_1x^{n-1}+e_2x^{n-2}+\cdots+ e_tx^{n-t}+\cdots+ e_n$. Suppose $p_r=\alpha_1^r+\alpha_2^r+\cdots+\alpha_n^r$  $(r=0, 1, \cdots )$. Then Newton-Girard identity says that
\begin{enumerate}
 \item 
If $r>n,  p_r+e_1p_{r-1}+e_2p_{r-2}+\cdots+p_1e_{r-1}+e_np_{r-n}=0$
 \item
If $r\leq n,  p_r+e_1p_{r-1}+e_2p_{r-2}+\cdots+p_1e_{r-1}+re_r=0.$	
 \end{enumerate}
 \end{corollary}	  
 \section{Proof of the theorems }
 In this section we prove the generalized power sum theorem. As a recipe to do so, let us describe some terminology. Let $A=\{a_1, a_2, \cdots, a_n\}$ be a finite set. Think $A$ to be the set of \emph{letters}. The \emph{freee monoid} $A^{*}$ is the set of all finite sequences of elements of $A,$ usually called \emph{words} with the operation of concatenation. Construct an algebra from $A^{*}$ by taking formal sum of elements of $A$ with coefficient in $\mathbb{Z},$ extending the multiplication by usual distributivity. For example, in this algebra,
 \begin{align*}
  (a_1+a_2)a_3&=a_1a_3+a_2a_3\\ (a_1+a_2)(a_1+a_2)&=a_1a_1+a_1a_2+a_2a_1+a_2a_2.
  \end{align*} 
 \begin{proof}[Proof of Theorem \ref{main thm}]
Let \[L=\{x_i^{(j)}, 1\leq j\leq r, 1\leq i\leq m\}\cup \{y_{\ell}, 2\leq \ell\leq m+1\}.\] Take $L$ to be the set of letters. Then the left hand side of Theorem \ref{main thm} can be interpreted as the sum of all words in $L^{*}$ of the form $x^{(1)}_{i_1}x^{(2)}_{i_2}\cdots x^{(r)}_{i_r}y_t,$ where each $i_p<t,$ for all $1\leq p\leq r.$ Let us call the word of this form, \emph{good word}, and let $G$ be the set of all good words. Now let us evaluate the sum of all good words in another way. For any $U\subset [m+1]$ with $|U|\geq 2,$ define $G_U=\{x^{(1)}_{i_1}x^{(2)}_{i_2} \cdots x^{(r)}_{i_r}y_t: \text{ each } i_p\in U\setminus \text{Max}(U), t=\text{ \text{Max}(U) }, \text{ and for any } u\in U\setminus\text{Max}(U), \text{there exists } q\in [r]: i_q=u \}.$ It is clear that \[G=\bigcup_{U\subset[m+1], |U|\geq 2} G_U.\] Now by the principle of inclusion and exclusion, the sum of all words in $G_U$  is \[\left(\sum\limits_{\varnothing \neq T\subset U\setminus \{\text{Max}(U)\}} (-1)^{|U|-|T|-1}\Pi_rT\right)y_{\text{Max(U)}}.\] This completes the proof. 

Now we proceed to the proof of Theorem \ref{colored nwtn thm}. Before getting into the proof we need the following:

Let $\Gamma$ be a digraph (not necessarily colored). A walk $w$ in $\Gamma$ is defined to be a sequence of vertices $w=v_0, v_1, \cdots, v_t$ such that for each $i\in [0, t-1]\cap \mathbb{Z},$ there is a directed  edge from $v_i$ to $v_{i+1}.$ Now let $w_1=v_0, v_1, \cdots, v_t$ and $w_2=v_t, v_{t_1}, \cdots, v_{t_k}$ be two walks  in $\Gamma.$ Then the \emph{concatenation} of $w_1$ and $w_2,$ denoted by $w_1\bigodot w_2,$ is the walk $v_0, v_1, \cdots, v_t, v_{t_1}, \cdots, v_{t_k}.$   
 \end{proof} 
\begin{proof}[Proof of Theorem \ref{colored nwtn thm}] 
First we prove the case $r>n$. To prove this, consider all ordered pairs $(w, \gamma)$, where $w$ is a colored closed walk and $\gamma$ is a colored linear subdigraph (possibly empty), such that $L(w)+L(\gamma)=r$ and $\text{ Col}(w)\cap \text{ Col}(\gamma)=\varnothing.$ Define the weight $W$ of $(w, \gamma)$ to be $W((w, \gamma))=(-1)^{c(\gamma)}W(w)W(\gamma)$. Note that the left hand side of $(1)$ in Theorem \ref{colored nwtn thm} is precisely equal to $\sum\limits_{(w, \gamma)}W((w, \gamma))$, where the summation runs over all ordered pairs $(w, \gamma)$ as described above.

Now the crucial observation is that,  since $r>n,$ either $w$ and $\gamma$ share a common vertex or $w$ is not a ``simple" closed walk (here simple means the graph structure of the closed walk is a directed cycle). Now take a particular pair $(w, \gamma)$ satisfying the above conditions. Suppose that $x$ is the initial and terminal vertex of $w$. Start moving from $x$ along $w$.  There are two possibilities: either, first we meet a vertex $y$ which is a vertex of $\gamma$ or, we complete a closed directed cycle $\acute{w}$ which is a subwalk of $w$ and during this journey from $x$ up to the completion of $\acute{w}$ we have not met any vertex of $\gamma$. Now if the first case holds, we form a new ordered pair $(\tilde{w}, \tilde{\gamma})$, where $\tilde{w}=\widehat{xy}|_{w}\bigodot \gamma_{y}\bigodot \widehat{yx}|_{w}$ and $\tilde{\gamma}=\gamma\setminus\{\gamma_y\}$, where $\widehat{xy}|_{w}$ is the walk from $x$ to $y$ along $w$ and $\gamma_y$ is the directed cycle of $\gamma$ containing the vertex $y$. Note that $W((\tilde{w}, \tilde{\gamma}))=-W((w, \gamma))$. Now if the second case holds, then form a new ordered pair $(\tilde{\tilde{w}}, \tilde{\tilde{\gamma}})$, where $\tilde{\tilde{w}}$ is formed by removing the directed cycle $\acute{w}$ from $w$ and $\tilde{\tilde{\gamma}}$ is $\gamma\cup \acute{w}$. Note also that $W((\tilde{\tilde{w}}, \tilde{\tilde{\gamma}}))=-W((w, \gamma))$. It is easy to see that,  this is in fact a sign reversing involution by additionally noting that $\text{ Col}(\tilde{w})\cap\text{ Col}(\tilde{\gamma})=\varnothing$ and $\text{ Col}(\tilde{\tilde{w}})\cap\text{ Col}(\tilde{\tilde{\gamma}})=\varnothing.$ This completes the first part of the proof.

Now we prove the case $r\leq n$. Let $B=\{ (w, \gamma):w$ is a colored closed walk of length $\geq 1, \gamma$ is a colored linear subdigraph (possibly empty), $L(w)+L(\gamma)=r \text{ and }\text{ Col}(w)\cap \text{ Col}(\gamma)=\varnothing\}.$ Consider the following sum $D=\sum\limits_{(w, \gamma)\in B}W((w, \gamma))+r\ell_{r, C}.$ Note that  the left hand side of $(2)$ in Theorem \ref{colored nwtn thm} is precisely equal to $D$.

Consider the subset of $B$ consisting of ordered pair $(w, \gamma)$ satisfying the conditions: either $w\cap \gamma\neq \varnothing$ or $w$ is not a simple closed walk. Call this subcollection \emph{BAD}. So the \emph{GOOD} members of $B$ are the ordered pairs $(w, \gamma)$ satisfying $w\cap \gamma=\varnothing$ and $w$ is a colored simple closed walk. Now observe that, the weights of the BAD members cancel among themselves just like the previous case (case, $r>n$). Now let us see,  how a GOOD member looks like. As a directed graph it is just a disjoint collection of distinct cycles with vertex set, say, $\{ v_1, v_2, \cdots, v_r\}$ i.e. it is a colored linear subdigraph $\dot{\gamma}$ with vertex set $\{v_1, v_2, \cdots, v_r\}$. Now for this $\dot{\gamma}$ with vertex set $\{v_1, v_2, \cdots, v_r\}$, we claim that there are precisely $r$ GOOD members $(w, \gamma)$. For the proof, take any vertex say $v_i$ from $\dot{\gamma}$. Consider the cycle $w$ in $\dot{\gamma}$ containing the vertex $v_i$. Let $\gamma_1=\dot{\gamma}\setminus \{ w\}$. Now the cycle $w$ can be thought of as a closed walk $w_{v_i}$ starting and ending at the vertex $v_i$. So we get a GOOD member $(w_{v_i}, \gamma_1)$. Since $v_i$ is arbitrary the claim follows.

The main observation is that the sum of the weights of all the GOOD members, found in this way from $\dot{\gamma}$ is $r(-1)^{c(\dot{\gamma})-1}w(\dot{\gamma})$. This cancels with the term $r(-1)^{c(\dot{\gamma})}w(\dot{\gamma})$ in the equation $D=\sum\limits_{(w, \gamma)\in B}W((w, \gamma))+r\ell_{r, C}$.	
\end{proof}
\begin{proof}[Proof of Theorem \ref{col newtn identity}] The proof immediately follows by applying Theorem \ref{colored nwtn thm} to the weighted colored digraph $\Gamma_{r, C}$ defined as follows:

The vertex set of the graph $\Gamma_{r, C}$ is $V(\Gamma_{r, C})=\{v_1, v_2, \cdots, v_n\}$ and $C=\{c_1, c_2, \cdots, c_r\}$ is the set of colors. For each vertex $v_j$ there are precisely $r$ directed edges $e^{(1)}_j, e_j^{(2)}, \cdots, e^{(r)}_j$ from $v_j$ to itself such that $e^{(k)}_j$ is colored with the color $c_k,$ for all $k=1, 2, \cdots, r.$ For $j\neq j^{'}$ there is no directed edge from $v_j$ to $v_{j^{'}}.$ Also for any $j,$ the edge from $v_j$ to itself with color $c_i,$ is given a weight $\alpha^{(i)}_j.$  	 
\end{proof}	    
\subsection*{Acknowledgement} The first author was supported by Department of Science and Technology grant EMR/2016/006624 and partly supported by  UGC Centre for Advanced Studies. Also the first author was supported by NBHM Post Doctoral Fellowship grant 0204/52/2019/RD-II/339. The second author was supported by NBHM Post Doctoral Fellowship grant 0204/3/2020/RD-II/2470.
\bibliographystyle{amsplain}
\bibliography{gen-inv-lcp}
\end{document}